\documentclass[a4paper,11pt]{article}
\usepackage{amsmath,amssymb}
\usepackage[T1]{fontenc}
\usepackage{latexsym}
\usepackage{pdfsync}
\usepackage{epsfig}
\usepackage{hyperref}
\usepackage{proof}
\usepackage[utf8]{inputenc}
\usepackage[a4paper, top=3cm, bottom=3cm ]{geometry}
\usepackage{graphicx}

\newtheorem{thm}{Theorem}[section]
\newtheorem{defn}{Definition}[section]
\newtheorem{lem}{Lemma}[section]

\title
{$L_p$ Markov exponent of certain UPC sets}

\author{\normalsize Tomasz Beberok \\
\small Faculty of Mathematics and Computer Science, Jagiellonian University,\\
\small Lojasiewicza 6, 30-048 Krakow, Poland \\}

\date{}

\begin{document}

\begin{center}
  \textbf{$L_p$ Markov exponent of certain UPC sets}
\end{center}
\vskip1em
\begin{center}
  Tomasz Beberok
\end{center}

\vskip2em

\noindent \textbf{Abstract.} In this paper we give sharp $L_p$ Markov type inequalities for derivatives of multivariate polynomials for a wide family of UPC domains.
\vskip1em


\section{Introduction}
\label{}
Let $\mathcal{P}_n(\mathbb{R}^m)$ be the class of all algebraic polynomials in $m$ variables with real coefficients of degree $n$. Further, let $C(\Omega)$ be the real space of all real valued continuous functions $f$ defined on a compact set $\Omega \subset \mathbb{R}^m$ with the norm $\|f\|_{C(\Omega)}:=\sup_{x \in \Omega} |f(x)|$, and let $L_{p}(\Omega)$, $1\leq p \leq \infty$, be the space of all Lebesgue-measurable functions $f$ on $\Omega \subset \mathbb{R}^m$ such that $$\|f\|_{L_{p}(\Omega)}:=(\int_{\Omega} |f(x)|^p  \, dx)^{1/p} < \infty \quad  \text{if} \quad 1\leq p < \infty,$$ and $L_{\infty}:=C(\Omega)$. Moreover, $\mathbb{N}=\{1,2,3,\ldots\}$ and $\mathbb{N}_0=\{0\} \cup \mathbb{N}$.
\begin{defn}
We say that a compact set  $\emptyset \neq E \subset \mathbb{R}^m$ satisfies $L_p$ Markov type inequality  (or: is a $L_p$ Markov set) if there exist $\kappa,C > 0$ such that, for each polynomial $P \in \mathcal{P} (\mathbb{R}^m)$ and each $\alpha \in \mathbb{N}_0^m$,
\begin{align}\label{MarkovLp}
  \left\| D^{\alpha} P \right\|_{L_p(E)}  \leq (C (\deg P)^{\kappa})^{| \alpha |} \|P\|_{L_p(E)},
\end{align}
where $D^{\alpha} P = \frac{ \partial^{| \alpha |} P  }{ \partial x_1^{\alpha_1} \ldots \partial x_m^{\alpha_m}   }$ and $| \alpha |= \alpha_1 + \cdots + \alpha_m$.
\end{defn}
Clearly, by iteration, it is enough to consider in the above definition multi-indices $\alpha$ with $|\alpha| = 1$. The inequality (\ref{MarkovLp}) is a generalization of the classical Markov inequality:
\begin{align*}
  \|P'\|_{C([-1,1])} \leq (\deg P)^2  \|P\|_{C([-1,1])}.
\end{align*}
The classical Markov inequality and its extensions for multivariate case play a central role in various approximation problems. \\ \indent
In this paper we shall consider the following problem:
\begin{center}
{\it For a given $L_p$ Markov set $E$ determine $\mu_p(E):=\inf\{\kappa : E \mbox{ satisfies } (\ref{MarkovLp}) \}$.}
\end{center}
The quantity $\mu_{p}(E)$ is called $L_p$ Markov exponent and was first considered by Baran and Ple\'{s}niak in \cite{BP} for $p=\infty$. For any compact set $E$ in $\mathbb{R}^m$ we have $\mu_p(E)\geq 2$. It is well known that for Locally Lipschitzian compact subsets of $\mathbb{R}^m$ and thus, in particular convex domains the $L_p$ Markov exponent is equal to $2$ (see \cite{G2}). If $E \subset \mathbb{R}^m$ is a $Lip\gamma$, $0<\gamma<1$ cuspidal domain, then  $\mu_{\infty}(E)=\frac{2}{\gamma}$ (see for instance, \cite{G1}, \cite{MB}, \cite{KS}). Recently, in \cite{Kroo}, for $1 \leq p <\infty$, the same exponent $\frac{2}{\gamma}$ as in $L_{\infty}$ norm was obtained for $Lip\gamma$, $0<\gamma<1$ cuspidal piecewise graph domains that are imbedded in an affine image of the $l_{\gamma}$ ball having one of its vertices on the boundary. Our goal is to establish $L_p$ Markov exponent of the following domains $$K:=\{(x,y) \in \mathbb{R}^2 : 0 \leq x \leq 1, \, ax^k \leq y \leq f(x)\},$$
where $k \in \mathbb{N}$, $k \geq 2$, $a>0$ and $f: [0,1] \rightarrow [0,\infty)$ is a convex function such that $f(1)>a$, $f'(0)=f(0)=0$, $f'(1)<\infty$, and $(f)^{1/k}$ is a concave function on the interval $(0,1)$. Since the domain $K$ is not a cuspidal piecewise graph domain, the $L_p$ Markov exponent cannot be obtained using the methods of \cite{Kroo}. In particular, $K$ has a cusp at the origin that cannot be connected to the interior of $K$ by a straight line. However, the point $(0,0)$ can be connected to the interior of $K$ by a polynomial curve. Thus the domain $K$ is a uniformly polynomially cuspidal subsets of $\mathbb{R}^2$ (briefly, UPC). Therefore, if $p=\infty$, the Markov exponent of $K$ for a wide family of convex functions is known (see \cite{MB},\cite{KS}). A very special case of $K$ was considered in \cite{TB} for $1 \leq p < \infty$.

\section{Auxiliary results}
In order to verify our main result we shall need some auxiliary statements. First we show that our set $K$ satisfies $L_p$ Markov type inequality, with exponent 2, for a derivative with respect to $x$.
\begin{lem}\label{lem1}
Let $f: [0,1] \rightarrow [0,\infty)$ be a convex function so that $f'(0)=f(0)=0$ and $f'(1)<\infty$. Suppose that there exist constants $a$ and $k$ such that $a>0$, $k \in \mathbb{N}$, $k \geq 2$, $f(1)>a$, and $(f)^{1/k}$ is a concave function on the interval $(0,1)$. Then, for every $1 \leq p < \infty$, there exists a positive constant $B_1$ such that, for each polynomial $P \in \mathcal{P} (\mathbb{R}^2)$,
\begin{align}\label{Ineqlem1}
  \left\| \frac{\partial P}{ \partial x} \right\|_{L_p(K)}  \leq B_1  (\deg P)^{2} \|P\|_{L_p(K)},
\end{align}
where $K:=\{(x,y) \in \mathbb{R}^2 : 0 \leq x \leq 1, \, ax^k \leq y \leq f(x)\}$.
\end{lem}
\begin{proof}
Let $x_0 \in (0,1)$. Define $K_{x_0}=\{(x,y) \in \mathbb{R}^2 : x_0 \leq x \leq 1, \, ax^k \leq y \leq f(x)\}$. Then the domain $K_{x_0}$ is locally Lipschitzian compact subsets of $\mathbb{R}^2$. Hence by the main result of \cite{G2}, there exists a positive constant $A$ such that
\begin{align*}
  \left\| \frac{\partial P}{ \partial x} \right\|_{L_p(K_{x_0})} \leq A (\deg P)^2 \left\| P \right\|_{L_p(K_{x_0})},
\end{align*}
for each polynomial $P \in \mathcal{P} (\mathbb{R}^2)$. Keeping $x_0$ fixed in $(0,1)$, according to the inequality above, it is enough to verify that there exists a positive constant $A_1$ such that
\begin{align*}
  \left\| \frac{\partial P}{ \partial x} \right\|_{L_p(K\setminus K_{x_0})} \leq A_1 (\deg P)^2 \left\| P \right\|_{L_p(K)},
\end{align*}
for each polynomial $P \in \mathcal{P} (\mathbb{R}^2)$.    The value of $x_0$ will be provided later. \newline \indent
Using the change of variables $y = s^k$, we have
\begin{align*}
  \left\| \frac{\partial P}{ \partial x} \right\|^p_{L_p(K)}= \int_{\hat{K}}  \left| \frac{\partial P}{ \partial x} (x,s^k) \right|^p ks^{k-1} \, dx ds,
\end{align*}
~where $\hat{K}=\{(x,s) \in \mathbb{R}^2 : 0 \leq x \leq 1, \, \sqrt[k]{a}x \leq s \leq (f(x))^{1/k}\}$.
\begin{figure}
  \centering
  \includegraphics[scale=0.6]{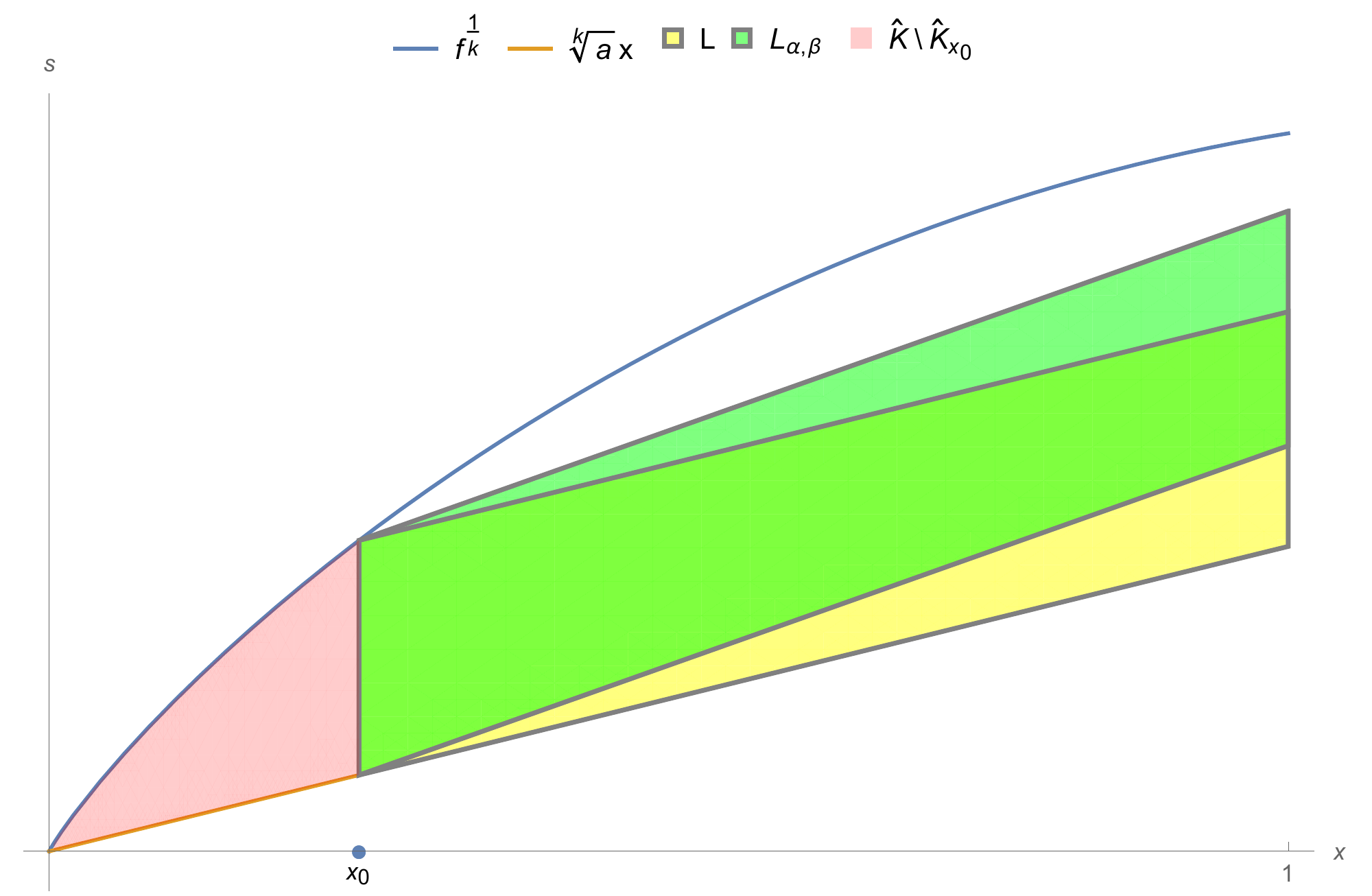}
  \caption{$L$, $L_{\alpha, \beta}$ and $\hat{K} \setminus \hat{K}_{x_0}$}\label{fig1}
\end{figure}

Since $f(1)>a$, we can select $\alpha, \beta$ so that $\sqrt[k]{a} < \alpha < \beta < (f(1))^{1/k}$. Now define $l_{\alpha, \beta}(x)=\alpha x + \beta - \alpha$. Since $(f(x))^{1/k}$ is a concave function  on the interval $(0,1)$, the equation $l_{\alpha, \beta}(x)=(f(x))^{1/k}$ has a unique solution in $(0,1)$. Let $x_0$ be the solution. Define
\begin{align*}
 &\hat{K}_{x_0}=\{(x,s) \in \mathbb{R}^2 : x_0 \leq x \leq 1, \, \sqrt[k]{a}x \leq s \leq (f(x))^{1/k}\} \\
 &L=\{(x,s) \in \mathbb{R}^2 : x_0 \leq x \leq 1, \, \sqrt[k]{a}x \leq s \leq \sqrt[k]{a}x + (f(x_0))^{1/k} - \sqrt[k]{a} x_0\}, \\
 &L_{\alpha, \beta}= \{(x,s) \in \mathbb{R}^2 : x_0 \leq x \leq 1, \, l_{\alpha, \beta}(x) -(f(x_0))^{1/k} + \sqrt[k]{a} x_0 \leq s \leq  l_{\alpha, \beta}(x)\}.
\end{align*}
An illustration of $L$, $L_{\alpha, \beta}$ and $\hat{K} \setminus \hat{K}_{x_0}$ is shown in Figure 1. It is clear that $L, L_{\alpha, \beta} \subset \hat{K}_{x_0} \subset \hat{K}$. Now define the function
\begin{align*}
  g(x)=\left\{
         \begin{array}{ll}
           (f(x))^{1/k}  \quad &x \in [0,x_0] \\
           \sqrt[k]{a}x + (f(x_0))^{1/k} - \sqrt[k]{a}x_0 \quad & x \in (x_0,1]
         \end{array}
       \right. ,
\end{align*}
\begin{align*}
   h(x)=\left\{
         \begin{array}{ll}
           (f(x))^{1/k}  \quad &x \in [0,x_0] \\
            l_{\alpha, \beta}(x)   \quad & x \in (x_0,1]
         \end{array}
       \right.
\end{align*}
and
\begin{align*}
j(x)= \left\{
         \begin{array}{ll}
              \sqrt[k]{a}x \quad &x \in [0,x_0] \\
            l_{\alpha, \beta}(x)  -(f(x_0))^{1/k} + \sqrt[k]{a} x_0 \quad & x \in (x_0,1]
         \end{array} .
       \right.
\end{align*}
Let $V_1=(\hat{K} \setminus \hat{K}_{x_0})  \cup L $ and $V_2=(\hat{K} \setminus \hat{K}_{x_0}) \cup L_{\alpha, \beta}  $. With this notation, we have
 \begin{align*}
   &\int_{V_1}  \left| \frac{\partial P}{ \partial x} (x,s^k) \right|^p ks^{k-1} \, dx ds= \int_{0}^{1} \int_{\sqrt[k]{a}x}^{g(x)}  \left| \frac{\partial P}{ \partial x} (x,s^k) \right|^p ks^{k-1} \, ds dx, \\
&\int_{V_2}  \left| \frac{\partial P}{ \partial x} (x,s^k) \right|^p ks^{k-1} \, dx ds= \int_{0}^{1} \int_{j(x)}^{h(x)}  \left| \frac{\partial P}{ \partial x} (x,s^k) \right|^p ks^{k-1} \, ds dx .
 \end{align*}
Making the change of variables $s=t+\sqrt[k]{a}x$ in the first and $s=t+\alpha x$ in the second integral and changing the orders of integration, we obtain
\begin{align*}
   &\int_{V_1}  \left| \frac{\partial P}{ \partial x} (x,s^k) \right|^p ks^{k-1} \, dx ds= \int_{0}^{ y_0} \int_{\tilde{g}(t)}^{1}  \left| \frac{\partial P}{ \partial x} (x,(t+\sqrt[k]{a}x)^k) \right|^p k(t+\sqrt[k]{a}x)^{k-1} \, dx dt, \\
&\int_{V_2}  \left| \frac{\partial P}{ \partial x} (x,s^k) \right|^p ks^{k-1} \, dx ds= \int_{\beta - \alpha -y_0}^{\beta - \alpha} \int_{\tilde{j}(t)}^{1}  \left| \frac{\partial P}{ \partial x} (x,(t+\alpha x)^k) \right|^p k(t+\alpha x)^{k-1} \, dx dt.
 \end{align*}
Here $y_0=(f(x_0))^{1/k}- \sqrt[k]{a}x_0$, $\tilde{g}$ is the inverse of the restriction of $(f(x))^{1/k}-\sqrt[k]{a}x$ to $[0,x_0]$
and
\begin{align*}
\tilde{j}(t)= \left\{
         \begin{array}{ll}
              \frac{t}{\sqrt[k]{a} - \alpha} \quad &t \in [\beta - \alpha -y_0,0] \\
            \tilde{h}(t) \quad & t \in (0,\beta - \alpha]
         \end{array} ,
       \right.
\end{align*}
where $\tilde{h}$ is the inverse of the restriction of $(f(x))^{1/k}-\alpha x$ to $(0,x_0]$. By Theorem 10.6 of \cite{Er}, there exists a positive constant $A_2$ (depending only on $x_0$ and $p$) such that
\begin{align*}
&  \int_{\tilde{g}(t)}^{1}  \left| Q'(x) \right|^p (t+\sqrt[k]{a}x)^{k-1} \, dx \leq A_2 (\deg Q +k)^{2p} \int_{\tilde{g}(t)}^{1}  \left| Q(x) \right|^p (t+\sqrt[k]{a}x)^{k-1} \, dx , \\
&  \int_{\tilde{j}(\tau)}^{1}  \left| R'(x) \right|^p (\tau+\alpha x)^{k-1} \, dx \leq A_2 (\deg R + k)^{2p}  \int_{\tilde{j}(\tau)}^{1}  \left| R(x) \right|^p (\tau+\alpha x)^{k-1} \, dx ,
 \end{align*}
for every $Q, R \in \mathcal{P}(\mathbb{R})$, for every $t \in [0,y_0]$ and for every $\tau \in [\beta - \alpha -y_0, \beta - \alpha]$. Hence
 \begin{align*}
   \int_{V_1}  \left| \frac{\partial P}{ \partial x} (x,s^k) + \sqrt[k]{a} ks^{k-1}\frac{\partial P}{ \partial y} (x,s^k) \right|^p & ks^{k-1} \, dx ds \\  &\leq A_2 (k \deg P +k)^{2p} \int_{V_1}  \left| P(x,s^k) \right|^p ks^{k-1} \, dx ds , \\
\int_{V_2}  \left| \frac{\partial P}{ \partial x} (x,s^k) + \alpha ks^{k-1}\frac{\partial P}{ \partial y} (x,s^k) \right|^p & ks^{k-1} \, dx ds \\ &\leq A_2 (k \deg P +k)^{2p} \int_{V_2}  \left| P(x,s^k) \right|^p ks^{k-1} \, dx ds.
 \end{align*}
Since $\hat{K} \setminus \hat{K}_{x_0} \subset V_1 \subset \hat{K}$ and $\hat{K} \setminus \hat{K}_{x_0} \subset V_2 \subset \hat{K}$, the above inequalities yield
 \begin{align*}
   \int_{\hat{K} \setminus \hat{K}_{x_0}}  \left| \frac{\partial P}{ \partial x} (x,s^k) \right|^p & ks^{k-1} \, dx ds \leq \frac {2^{p-1} A_2}{(\alpha - \sqrt[k]{a})^p} (k \deg P +k)^{2p} \int_{\hat{K}}  \left| P(x,s^k) \right|^p ks^{k-1} \, dx ds.
 \end{align*}
Therefore,
\begin{align*}
  \left\| \frac{\partial P}{ \partial x} \right\|^p_{L_p(K\setminus K_{x_0})} \leq \frac {2^{p-1} A_2}{(\alpha - \sqrt[k]{a})^p} (k \deg P +k)^{2p}  \left\| P \right\|^p_{L_p(K)}.
\end{align*}
Thus there is a positive constant $A_1$ such that
\begin{align*}
  \left\| \frac{\partial P}{ \partial x} \right\|_{L_p(K \setminus K_{x_0})} \leq A_1 (\deg P)^2 \left\| P \right\|_{L_p(K)},
\end{align*}
for each polynomial $P \in \mathcal{P} (\mathbb{R}^2)$.
\end{proof}
\newline \indent
The second lemma shows Remez-type inequality on $K$.
\begin{lem}\label{lem2}
Let $f,a,k,$ and $K$ be as in Lemma \ref{lem1}. Then, for every $1 \leq p < \infty$, there exists a positive constant $B_2$ such that, for every $n \in \mathbb{N}$, $n \geq 2$ and each polynomial $P \in \mathcal{P}_n (\mathbb{R}^2)$,
\begin{align}\label{Ineqlem2}
  \left\| P \right\|_{L_p(K)}  \leq B_2  {\|P\|}_{L_p(K_{1/n^2})},
\end{align}
where $K_{1/n^2}:=\{(x,y) \in \mathbb{R}^2 : \frac{1}{n^2} \leq x \leq 1, \, ax^k \leq y \leq f(x)\}$.
\end{lem}
\begin{proof}
Proceeding as in the proof of Lemma \ref{lem1}, we have
\begin{align*}
  \left\| P \right\|^p_{L_p(K)}= \int_{\hat{K}}  \left| P (x,s^k) \right|^p ks^{k-1} \, dx ds,
\end{align*}
where $\hat{K}=\{(x,s) \in \mathbb{R}^2 : 0 \leq x \leq 1, \, \sqrt[k]{a}x \leq s \leq (f(x))^{1/k}\}$. If we set
$$u(x)=\left\{
         \begin{array}{ll}
            (f(x))^{1/k}   \quad &x \in [0,1/4] \\
             \sqrt[k]{a}x + (f(1/4))^{1/k} + \frac{\sqrt[k]{a}}{4}  \quad & x \in (1/4,1]
         \end{array}
       \right.   $$
and
\begin{align*}
  \tilde{K}=\{(x,s) \in \mathbb{R}^2 : 0 \leq x \leq 1, \, \sqrt[k]{a}x \leq s \leq u(x)\} ,
\end{align*}
we can write
\begin{align*}
  \int_{\tilde{K}}  \left| P (x,s^k) \right|^p ks^{k-1} \, dx ds= \int_{0}^{1} \int_{\sqrt[k]{a}x}^{u(x)}  \left| P (x,s^k) \right|^p ks^{k-1} \, ds dx.
\end{align*}
Using the change of variables $s = t + \sqrt[k]{a}x $ and changing the order of integration, we have
\begin{align*}
  \int_{\tilde{K}}  \left| P (x,s^k) \right|^p ks^{k-1} \, dx ds= \int_{0}^{\eta} \int_{\tilde{u}(t)}^{1} \left| P (x,s^k) \right|^p k(t + \sqrt[k]{a}x )^{k-1} \, dx dt,
\end{align*}
where $\eta=(f(1/4))^{1/k} + \frac{\sqrt[k]{a}}{4}$ and $\tilde{u}$ is the inverse of the restriction of $(f(x))^{1/k}-\sqrt[k]{a}x$ to $[0,1/4]$. By Theorem 4.5 of \cite{Er}, there is an absolute constant $B_2>0$ such that
\begin{align*}
  \int_{\tilde{u}(t)}^{1} \left| P (x,s^k) \right|^p k(t + \sqrt[k]{a}x )^{k-1} \, dx \leq B^p_2 \int_{\tilde{u}_n(t)}^{1} \left| P (x,s^k) \right|^p k(t + \sqrt[k]{a}x )^{k-1} \, dx
\end{align*}
for every $n \in \mathbb{N}$, $n \geq 2$ and each polynomial $P \in \mathcal{P}_n (\mathbb{R}^2)$. Here $\tilde{u}_n$ is the function on $[0,\eta]$ defined by  $\tilde{u}_n(t)=\max\{1/n^2,\tilde{u}(t)\}$. Hence
\begin{align*}
  \int_{\tilde{K}}  \left| P (x,s^k) \right|^p ks^{k-1} \, dx ds \leq B^p_2  \int_{\tilde{K}_{1/n^2}}  \left| P (x,s^k) \right|^p ks^{k-1} \, dx ds,
\end{align*}
where $\tilde{K}_{1/n^2}=\{(x,s) \in \mathbb{R}^2 : 1/n^2 \leq x \leq 1, \, \sqrt[k]{a}x \leq s \leq u(x)\}$. Therefore
\begin{align*}
  \int_{\hat{K}}  \left| P (x,s^k) \right|^p ks^{k-1} \, dx ds \leq B^p_2  \int_{\hat{K}_{1/n^2}}  \left| P (x,s^k) \right|^p ks^{k-1} \, dx ds,
\end{align*}
which completes the proof.
\end{proof}
\section{Main result}
\begin{thm}\label{mainthm}
Let $f: [0,1] \rightarrow [0,\infty)$ be a convex function so that $f'(0)=f(0)=0$ and $f'(1)<\infty$. Suppose that there exist constants $a$ and $k$ such that $a>0$, $k \in \mathbb{N}$, $k \geq 2$, $f(1)>a$, and $(f)^{1/k}$ is a concave function on the interval $(0,1)$. Let $K=\{(x,y) \in \mathbb{R}^2 : 0 \leq x \leq 1, \, ax^k \leq y \leq f(x)\}$. Then, for every $1 \leq p < \infty$,
 \begin{align}\label{main}
    \mu_{p}(K) = \inf \{ \tau >0 : \, \exists_{C>0} \, \forall_{n \in \mathbb{N}} \, \, n^2 \leq  Cf'(1/n^2) n^\tau \}.
  \end{align}
\end{thm}
\begin{proof}
First we show that $\mu_{p}(K) \leq \inf \{ \tau >0 : \, \exists_{C>0} \, \forall_{n \in \mathbb{N}} \, \, n^2 \leq  Cf'(1/n^2) n^\tau \}$. By Lemma \ref{lem1} and Lemma \ref{lem2}, it suffices to prove that there exists a positive constant $C$ such that
\begin{align}\label{res}
  \left\| \frac{\partial P}{ \partial y} \right\|_{L_p(K_{1/n^2})}  \leq C \frac{n^2}{f'(1/n^2)} \|P\|_{L_p(K)},
\end{align}
for every polynomial $P \in \mathcal{P}_n (\mathbb{R}^2)$. Let $\eta_n'=f'(1/n^2)$ and $\eta_n=f(1/n^2)-\frac{f'(1/n^2)}{n^2}$. Since $f(1)>a$, we can find a constant $c$ such that $f(1)>c^{k+1}f(1)>a$. Let
\begin{align*}
  K_{1/n^2}^c=\{ (x,y) \in \mathbb{R}^2 : 1/n^2 \leq x \leq 1, \, ax^k + (1-c)(\eta_n' x + \eta_n) \leq y \leq f(x)  \}.
\end{align*}
Since $(f)^{1/k}$ is a concave function, the set $K_{1/n^2}^c$ is nonempty. Using the change of variables $y=s^k+(1-c)(\eta_n' x + \eta_n)$, we have
\begin{align}\label{K}
  \left\| P\right\|^p_{L_p(K_{1/n^2}^c)}= \int_{\hat{K}_{1/n^2}^c}  \left| P(x,s^k+(1-c)(\eta_n' x + \eta_n)) \right|^p ks^{k-1} \, dx ds,
\end{align}
where $\hat{K}_{1/n^2}^c=\{ (x,s) \in \mathbb{R}^2 : 1/n^2 \leq x \leq 1, \, \sqrt[k]{a}x \leq s \leq [f(x)-(1-c)(\eta_n' x + \eta_n)]^{1/k}\}$. It now follows from the convexity of $f$ that
 \begin{align*}
    & \hat{K}_{1/n^2}^c \subset \hat{K}_{1/n^2}=\{ (x,s) \in \mathbb{R}^2 : 1/n^2 \leq x \leq 1, \, \sqrt[k]{a}x \leq s \leq (f(x))^{1/k}\},  \\
   & \{ (x,s) \in \mathbb{R}^2 : 1/n^2 \leq x \leq 1, \, \sqrt[k]{a}x \leq s \leq (cf(x))^{1/k}\} \subset \hat{K}_{1/n^2}^c.
 \end{align*}
Select $\alpha, \beta$ so that $c(cf(1))^{1/k}> \beta > \alpha > \sqrt[k]{a}$. Let $l_{\sqrt[k]{a}, \alpha}(x)=\sqrt[k]{a}x + \alpha - \sqrt[k]{a}$. Then the equation $l_{\sqrt[k]{a}, \alpha}(x)=(cf(x))^{1/k}$ has a unique solution in $[0,1]$, and we denote it by $x_1$. Let $\lambda=\frac{\beta - (cf(x_1))^{1/k}}{1-x_1}$, and define $l_{\lambda,\beta}(x)=\lambda x + \beta - \lambda$. Now we wish to prove that the equation $l_{\lambda,\beta}(x)=(f(x))^{1/k}$ has a unique solution. Since $1>c>0$,
\begin{align}\label{1}
(f(1))^{1/k}>c(cf(1))^{1/k}>\beta=l_{\lambda,\beta}(1).
\end{align}
By the concavity of $(cf(x))^{1/k}$, and by the fact that $f(0)=0$, we have
\begin{align*}
  (cf(x))^{1/k} \geq (cf(1))^{1/k}x \quad \text{for every} \quad x \in [0,1].
\end{align*}
Thus by $(cf(1))^{1/k} > \beta$,
\begin{align*}
    (cf(1))^{1/k}x > \beta x \quad \text{for every} \quad x \in (0,1].
\end{align*}
Therefore
\begin{align}\label{0}
 l_{\lambda,\beta}(0)= \beta - \lambda =  \frac{(cf(x_1))^{1/k}-\beta x_1}{1-x_1} >0=(f(0))^{1/k}.
\end{align}
\begin{figure}
  \centering
  \includegraphics[scale=0.5]{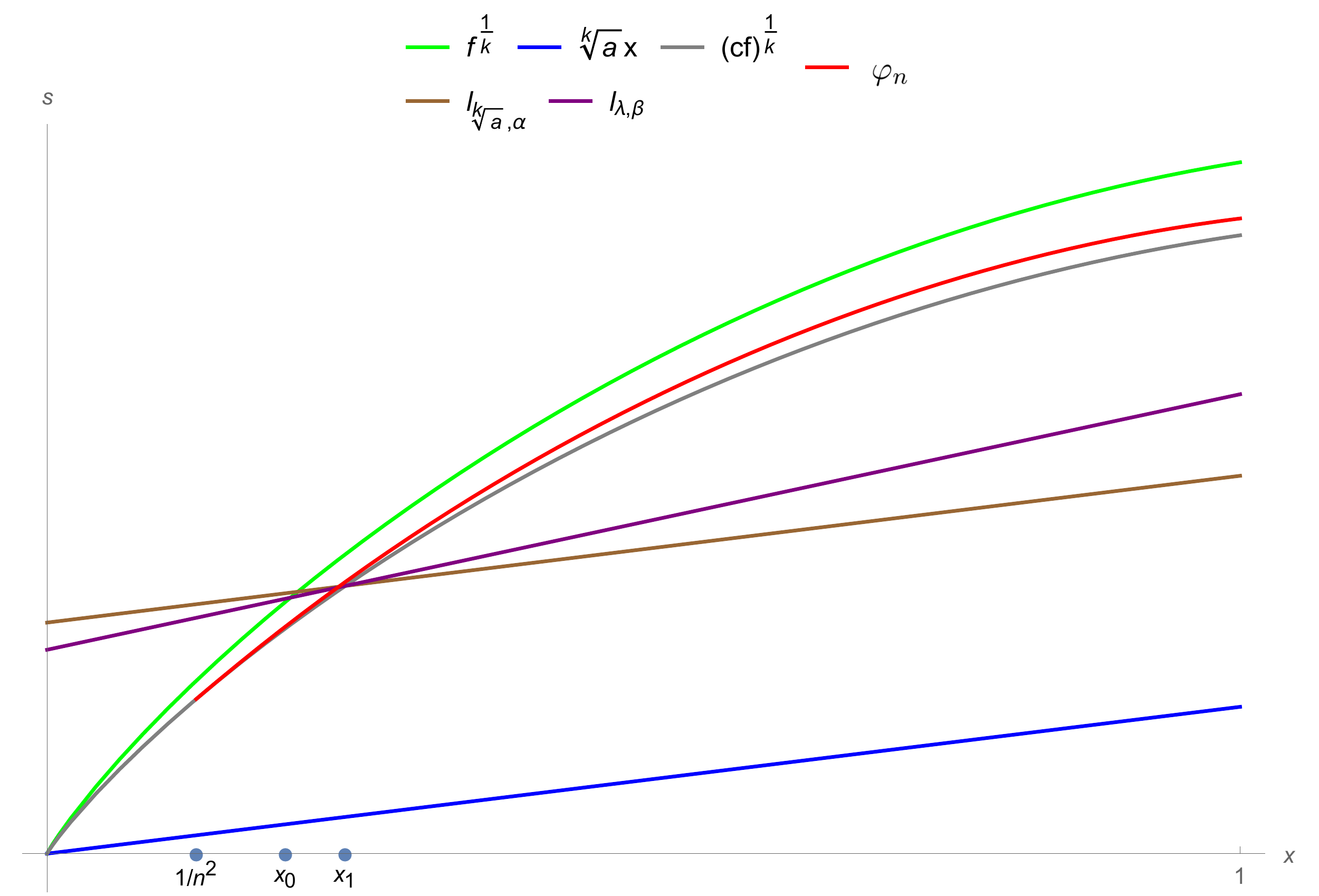}
  \caption{Graphs of $f^{1/k}$, $\sqrt[k]{a}x$, $(cf)^{1/k}$, $l_{\sqrt[k]{a},\beta}$, $l_{\lambda,\beta}$  and $\varphi_n$ }\label{fig2}
\end{figure}
By (\ref{1}) and (\ref{0}) there exists $x_0 \in (0,1)$ such that $l_{\lambda,\beta}(x_0)=(f(x_0))^{1/k}$. The Lemma 1.1.4 from \cite{NP} asserts the uniqueness of $x_0$. Let $\varphi_n(x):=[f(x)-(1-c)(\eta_n' x + \eta_n)]^{1/k}$ for every $1/n^2 < x_0$ and  $x \in [1/n^2,1]$. The functions $f^{1/k}$, $\sqrt[k]{a}x$, $(cf)^{1/k}$, $l_{\sqrt[k]{a},\beta}$, $l_{\lambda,\beta}$,  and $\varphi_n$ are illustrated in  Figure \ref{fig2}. In order to be able to implement the techniques used in the proof of Lemma \ref{lem1}, we must show that $\varphi_n(x) \geq   \lambda x + \varphi_n(1/n^2) - \lambda/n^2$ for $x \in [x_0,1]$ and the function $\varphi_n(x) - \lambda x$ is strictly increasing on $[1/n^2,x_0]$. It is clear that
\begin{align*}
  \varphi_n(x) \geq   \lambda x + \varphi_n(1/n^2) - \lambda/n^2 \quad \text{for} \quad x \in [x_1,1].
\end{align*}
Since $c(cf(1))^{1/k}> \beta$,
\begin{align}\label{lambda}
  \frac{c(cf(1))^{1/k} - c(cf(x_1))^{1/k}}{1-x_1} > \frac{c(cf(1))^{1/k} - (cf(x_1))^{1/k}}{1-x_1} > \lambda.
\end{align}
The concavity of $f^{1/k}$ yields
\begin{align}\label{lambda1}
   c^{1/k+1} (f^{1/k})'(x)  \geq \frac{c(cf(1))^{1/k} - c(cf(x_1))^{1/k}}{1-x_1} \quad \text{for every} \quad x \in (0,x_1].
\end{align}
The convexity of $f$ gives
\begin{align}\label{lambda2}
  \varphi_n'(x)  \geq  c^{1/k+1} (f^{1/k})'(x)  \quad \text{for every} \quad x \in (1/n^2,1].
\end{align}
Putting the inequalities $(\ref{lambda})-(\ref{lambda2})$ together, we have
\begin{align*}
  \varphi_n'(x) > \lambda \quad \text{for every} \quad x \in [1/n^2,x_1].
\end{align*}
Therefore $\varphi_n(x) - \lambda x$ is strictly increasing on $[1/n^2,x_0]$ and
\begin{align*}
  \varphi_n(x) \geq   \lambda x + \varphi_n(1/n^2) - \lambda/n^2 \quad \text{for} \quad x \in [1/n^2,x_1].
\end{align*}
Now, similarly as in the proof of Lemma \ref{lem1}, we can establish that there exists a positive constant $B$ (independent on $n$) such that
\begin{align*}
   \int_{\hat{K}_{1/n^2,x_0}^c}  & \left| \frac{\partial P}{ \partial x}  (x,z) + \sqrt[k]{a} (ks^{k-1} + (1-c)\eta_n') \frac{\partial P}{ \partial y} (x,z) \right|^p  ks^{k-1} \, dx ds \\  &\leq B (k \deg P +k)^{2p} \int_{\hat{K}_{1/n^2}^c}  \left| P(x,z) \right|^p ks^{k-1} \, dx ds , \\
\int_{\hat{K}_{1/n^2,x_0}^c}  & \left| \frac{\partial P}{ \partial x} (x,z) + \lambda (ks^{k-1} + (1-c)\eta_n') \frac{\partial P}{ \partial y} (x,z) \right|^p  ks^{k-1} \, dx ds \\  &\leq B (k \deg P +k)^{2p} \int_{\hat{K}_{1/n^2}^c}  \left| P(x,z) \right|^p ks^{k-1} \, dx ds,
 \end{align*}
where $z=s^k+(1-c)(\eta_n' x + \eta_n)$ and $$\hat{K}_{1/n^2,x_0}^c=\{ (x,s) \in \mathbb{R}^2 : 1/n^2 \leq x \leq x_0, \, \sqrt[k]{a}x \leq s \leq [f(x)-(1-c)(\eta_n' x + \eta_n)]^{1/k}\}.$$
\end{proof}
From these inequalities, we see that
\begin{align*}
   \int_{\hat{K}_{1/n^2,x_0}^c}  & \left| \frac{\partial P}{ \partial x}  (x,z) + (1-c)\eta_n' \frac{\partial P}{ \partial y} (x,z) \right|^p  ks^{k-1} \, dx ds \\  &\leq B \frac{2^{p-1} (k \deg P +k)^{2p}}{(\lambda - \sqrt[k]{a})^p} \int_{\hat{K}_{1/n^2}^c}  \left| P(x,z) \right|^p ks^{k-1} \, dx ds.
 \end{align*}
Hence by (\ref{K}), we have
\begin{align*}
   \int_{K_{1/n^2,x_0}^c}   \left| \frac{\partial P}{ \partial x}  (x,y) + (1-c)\eta_n' \frac{\partial P}{ \partial y} (x,y) \right|^p &  \, dx dy \\  \leq & B \frac{2^{p-1} (k \deg P +k)^{2p}}{(\lambda - \sqrt[k]{a})^p} \int_{K_{1/n^2}^c}  \left| P(x,y) \right|^p  \, dx dy.
 \end{align*}
Thus, by Lemma \ref{lem1}, and by the fact that $K_{1/n^2,x_0}^c \subset K$, we have
 \begin{align}\label{part1}
   \int_{K_{1/n^2,x_0}^c}   \left| (1-c)\eta_n' \frac{\partial P}{ \partial y} (x,y) \right|^p &  \, dx dy \nonumber \\  \leq & (B +B_1) \frac{4^{p-1} (k \deg P +k)^{2p}}{(\lambda - \sqrt[k]{a})^p} \int_{K}  \left| P(x,y) \right|^p  \, dx dy.
 \end{align}
\indent Now we wish to prove a similar result to (\ref{part1}) for the following domain
\begin{align*}
   K_{0,\delta}^a=\{ (x,y) \in \mathbb{R}^2 : 0 \leq x \leq \delta, \, ax^k \leq y \leq ax^k + (1-c)(\eta_n' x + \eta_n) \}.
\end{align*}
where $\delta$ is a fixed number in $(0,1)$ which does not depend on $n$. Using the change of variables $y =s + ax^k$, we have
\begin{align*}
\int_{K_{0,\delta}^a}   \left|  \frac{\partial P}{ \partial y} (x,y) \right|^p  \, dx dy = \int_{T_{n,\delta}}  \left|  \frac{\partial P}{ \partial y} (x,s + ax^k) \right|^p  \, dx ds,
\end{align*}
where $T_{n,\delta}=\{ (x,s) \in \mathbb{R}^2 : 0 \leq x \leq \delta, \, 0 \leq s \leq (1-c)(\eta_n' x + \eta_n) \}.$ It is clear that
\begin{align}\label{Para}
  T_{n,1/2} \subset \mathrm{L}:=\{(x,s) \in \mathbb{R}^2 : 0 \leq s \leq \sigma, \, \psi_n(s) \leq x \leq  \psi_n(s) + 1/2 \} \subset T_{n,1},
\end{align}
where $\psi_n(s)=\frac{s}{(1-c)\eta_n'} - \frac{\eta_n}{\eta_n'}$ and $\sigma=(1-c)(\frac{\eta_n'}{2} + \eta_n)$. A change of variables gives
\begin{align*}
 \int_{\mathrm{L}}  \left|  \frac{\partial P}{ \partial y} (x,s + ax^k) \right|^p  \, dx ds= \int_{0}^{\sigma} \int_{0}^{ \frac{1}{2} } \left|  \frac{\partial P}{ \partial y} (t + \psi_n(s), s + a(t + \psi_n(s))^k) \right|^p \, dt ds.
\end{align*}
Applying the univariate $L^p$ Markov inequality twice we obtain, with an absolute constant $B_3$,
\begin{align*}
  \int_{\mathrm{L}}  \left|  \frac{\partial P}{ \partial y} (x,s + ax^k) + (\psi_n'(s)-1) \frac{\partial P}{ \partial x} (x,s + ax^k) \right|^p & \, dx ds \\  \leq (B_3(1+1/\sigma))^p (k \deg P)^{2p} \int_{\mathrm{L}} & \left|  P (x,s + ax^k) \right|^p  \, dx ds.
\end{align*}
Thus, by (\ref{Para}), and by the fact that $K_{0,1}^a \subset K$, we have
\begin{align}\label{1/2y}
  \int_{K_{0,1/2}^a}   \left|  \frac{\partial P}{ \partial y} (x,y) +  (\psi_n'(s)-1) \frac{\partial P}{ \partial x} (x,y)\right|^p & \, dx dy \nonumber \\ \leq (B_3(1+1/\sigma))^p & (k \deg P)^{2p}  \int_{K}    \left|  P (x,y) \right|^p  \, dx dy
\end{align}
By Lemma \ref{lem1},
 \begin{align}\label{1/2x}
  \int_{K_{0,1/2}^a}   \left| (\psi_n'(s)-1) \frac{\partial P}{ \partial x} (x,y)\right|^p  \, dx dy  \leq (B_1|\psi_n'(s)-1|)^p  (\deg P)^{2p}  \int_{K}    \left|  P (x,y) \right|^p  \, dx dy.
\end{align}
By (\ref{1/2y}) and (\ref{1/2x}), there exists a positive constant $B_4$ such that
\begin{align}\label{dy}
   \int_{K_{0,1/2}^a}   \left|  \frac{\partial P}{ \partial y} (x,y) \right|^p  \, dx dy \leq \left(\frac{B_4}{\eta_n'} \right)^p (\deg P)^{2p}  \int_{K}    \left|  P (x,y) \right|^p  \, dx dy.
\end{align}
Now by (\ref{part1}) and (\ref{dy}), there exists a positive constant $B_5$ such that
 \begin{align}\label{dy2}
   \int_{K_{1/n^2,\theta}}   \left|  \frac{\partial P}{ \partial y} (x,y) \right|^p  \, dx dy \leq \left(\frac{B_5}{\eta_n'} \right)^p (\deg P)^{2p}  \int_{K}    \left|  P (x,y) \right|^p  \, dx dy,
\end{align}
where $\theta:=\min\{x_0, 1/2\}$ and
$K_{1/n^2,\theta}=\{(x,y) \in \mathbb{R}^2 : 1/n^2 \leq x \leq \theta, \, ax^k \leq y \leq f(x)\}.$
Observe that $K_{\theta,1}=\{(x,y) \in \mathbb{R}^2 : \theta \leq x \leq 1, \, ax^k \leq y \leq f(x)\}$ is locally Lipschitzian compact subsets of $\mathbb{R}^2$, then there exists a positive constant $B_6$, independent of $n$,  such that
\begin{align}\label{dy3}
   \int_{K_{\theta,1}}   \left|  \frac{\partial P}{ \partial y} (x,y) \right|^p  \, dx dy \leq B^p_6  (\deg P)^{2p}  \int_{K}    \left|  P (x,y) \right|^p  \, dx dy.
\end{align}
Thus, from inequalities (\ref{dy2}) and (\ref{dy3}), we conclude (\ref{res}). \\
\indent To prove that
\begin{align*}
  \mu_{p}(K) \geq \inf \{ \tau >0 : \, \exists_{C>0} \, \forall_{n \in \mathbb{N}} \, \, n^2 \leq  Cf'(1/n^2) n^\tau \},
\end{align*}
let $U_n(x,y)=yP_n^{(\omega,\sigma)}(1-x)$. Here $P_n^{(\omega,\sigma)}$ denotes the Jacobi polynomial of degree $n$ associated with parameters $\omega$, $\sigma$. Then
\begin{align*}
   {\left\| U_n \right\|}^p_{L_p(K_{1/n^2})}=  \frac{1}{p+1}  \int_{ \frac{1}{n^2} }^{1} \left((f(x))^{p+1} - (ax^k)^{p+1}\right) |P_n^{(\omega,\sigma)}(1-x)|^p   \, dx.
\end{align*}
Therefore,
\begin{align*}
   {\left\| U_n \right\|}^p_{L_p(K_{1/n^2})} \leq  \frac{1}{p+1}  \int_{ \frac{1}{n^2} }^{1} (f(x))^{p+1}  |P_n^{(\omega,\sigma)}(1-x)|^p   \, dx.
\end{align*}
Applying certain properties of Jacobi polynomials $P_n^{(\omega,\sigma)}$
 verified in \cite{Szego}, (7.32.5), p. 169, and using $k f(x) \geq x f'(x)$, one can show that there exists $\Lambda >0$ such that
\begin{align*}
    \frac{1}{p+1}  \int_{ \frac{1}{n^2} }^{1} (f(x))^{p+1}  |P_n^{(\omega,\sigma)}(1-x)|^p   \, dx \leq \Lambda (\eta_n')^{p+1} n^{\omega p -4 - 2p}
\end{align*}
whenever  $\omega p + p/2 -2 > 2k(p+1)$. See \cite{TB} for details. Thus,
\begin{align}\label{oneway}
  {\left\| U_n \right\|}^p_{L_p(K_{1/n^2})} \leq \Lambda (\eta_n')^{p+1} n^{\omega p -4 - 2p}.
\end{align}
\indent On the other hand, by the definition of $K$,
 \begin{align}\label{revway1}
 {\left\| \frac{\partial U_n}{ \partial y} \right\|}^p_{L_p(K)}=\int_{K}   |P_n^{(\omega,\sigma)}(1-x)|^p  \, dxdy \geq \int_{0}^{\frac{1}{n^2}} \int_{ax^k}^{f(x)}   |P_n^{(\omega,\sigma)}(1-x)|^p  \, dy dx.
 \end{align}
Since $cf(x) \geq a x^k $, we can write
\begin{align}\label{revway2}
  \int_{0}^{\frac{1}{n^2}} \int_{ax^k}^{f(x)}   |P_n^{(\omega,\sigma)}(1-x)|^p  \, dy dx \geq  \int_{0}^{\frac{1}{n^2}} (1-c)f(x)  |P_n^{(\omega,\sigma)}(1-x)|^p  \, dx.
 \end{align}
By making the change of variable $x=\frac{z}{2n^2}$, we obtain
 \begin{align}\label{m3e21}
 \int_{0}^{\frac{1}{n^2}} f(x)  |P_n^{(\omega,\sigma)}(1-x)|^p  \, dx = \frac{1}{2n^2}\int_{0}^{2} f(\tau_n(z)) |P_n^{(\omega,\sigma)}(1-\tau_n(z))|^p  \, dz,
 \end{align}
 where $\tau_n(z)=\frac{z}{2n^2}$.  By the formula of Mehler-Heine type (see \cite{Szego}, Theorem 8.1.1.)
\begin{align*}
 \frac{1}{2n^2}  |P_n^{(\omega,\sigma)}(1-\tau_n(z))|^p   \geq \frac{n^{\omega p}}{4^pn^2} (4(z/2)^{-\omega} J_{\omega}(z) - 1/\Gamma(\omega+2))^p
 \end{align*}
for $\omega >0$, and all sufficiently large $n$. Here $J_{\omega}(z)$ is the Bessel functions of the first kind. Since
\begin{align*}
\min_{z \in [0,2]} \{ (z/2)^{-\omega} J_{\omega}(z)\} \geq \min_{z \in [0,2]} \left\{ \frac{1}{\Gamma(\omega+1)} - \frac{z^2}{4\Gamma(\omega+2)} \right\}=\frac{\omega}{\Gamma(\omega+2)},
 \end{align*}
we have
\begin{align}\label{m3e24}
 \frac{1}{2n^2}\int_{0}^{2} f(g_n(z))  |P_n^{(\omega,\sigma)}(1-g_n(z))|^p  \, dz \geq \left(\frac{4\omega-1}{4\Gamma(\omega+2)}\right)^p n^{\omega p-2}  \int_{0}^{2} f(g_n(z))  \, dz.
 \end{align}
Then integration by parts shows that
 \begin{align*}
 \int_{0}^{2} f(g_n(z))   \, dz =   2f(1/n^2)  - \frac{1}{2n^2} \int_{0}^{2} z f'(g_n(z)) \, dz.
 \end{align*}
Since $kf(x) \geq xf'(x)$, this leads to
 \begin{align}\label{m3e26}
 \int_{0}^{2} f(g_n(z))   \, dz \geq   \frac{2f(1/n^2)}{k+1}  \geq \frac{2\eta_n'}{n^{2}k(k+1)}.
 \end{align}
From (\ref{revway1})-(\ref{m3e26}) we see that
 \begin{align*}
 {\left\| \frac{\partial U_n}{ \partial y} \right\|}^p_{L_p(K)}=\int_{K}   |P_n^{(\omega,\sigma)}(1-x)|^p  \, dtdy \geq \left(\frac{4\omega-1}{4\Gamma(\omega+2)}\right)^p \frac{2 n^{\omega p-4}\eta_n'}{k(k+1)}.
 \end{align*}
This last inequality together with (\ref{oneway}) imply that there exists a positive constant $\Theta$ depending only on $f$, $k$  and $\omega$ such that, for each $n$,
\begin{align*}
 {\left\| \frac{\partial U_n}{ \partial y} \right\|}_{L_p(K)} \geq \Theta \frac{n^2}{\eta_n'} {\left\| U_n \right\|}_{L_p(K_{1/n^2})}.
 \end{align*}
An application of Lemma \ref{lem2} completes the proof.
\section{Concluding remarks}
At the end of this article, we would like to make some comments related to the techniques used in this work and the relationship between the results and the known results. In addition, we will provide some consequences of the results obtained.
\begin{itemize}
  \item Let $k \in \mathbb{N}$, $k \geq 2$, and $0<a<1$. If $1<r \leq k$, then we can apply the main result (Theorem \ref{mainthm}) to the following domains:
      \begin{itemize}
        \item[$\ast$] $D_1=\{(x,y) \in \mathbb{R}^2 : 0 < x \leq 1, \quad ax^k \leq y \leq b_1x^r \}$,
        \item[$\ast$] $D_2=\{(0,0)\} \cup \{(x,y) \in \mathbb{R}^2 : 0 < x \leq 1, \quad ax^k \leq y \leq x^r \ln (-\ln ( b_2 x))\}$,
        \item[$\ast$] $D_3=\{(0,0)\} \cup \{(x,y) \in \mathbb{R}^2 : 0 < x \leq 1, \quad ax^k\leq y \leq -x^r \ln (b_3 x)\}$,
        \item[$\ast$] $D_4=\{(0,0)\} \cup \{(x,y) \in \mathbb{R}^2 : 0 < x \leq 1, \quad ax^k\leq y \leq x^r (-\ln (b_4 x))^c\}$,
      \end{itemize}
    for  appropriately adjusted constants $b_1,b_2,b_3,b_4$ and $c$. In particular, the $L^p$ Markov exponent of all of the above domains is equal to $2r$. However, $D_1$ satisfies $L^p$ Markov type inequality with exponent $2r$ unlike other domains.
  \item The techniques used in this paper can be applied to the more general domains. Let $1>d>0$. Let $f,g : [0,1] \rightarrow \mathbb{R}$ be continuous functions so that $f'(0)=f(0)=g(0)=0$. Suppose that there exist constants $a$ and $k$ such that $a>0$, $k \in \mathbb{N}$, $k \geq 2$,  $(f)^{1/k}$ is concave on the interval $(0,d)$, $f$ and $(g)^{1/k}$ are convex on the interval $(0,d)$, and  $f(x) > ax^k > g(x)$ for $x \in (0,d)$. If $E_{d}:=\{ (x,y) \in \mathbb{R}^2 : d \leq x \leq 1, \, g(x) \leq y \leq f(x) \}$ is locally Lipschitzian, then
      $$ \mu_{p}(E) = \inf \{ \tau >0 : \, \exists_{C>0} \, \forall_{n \in \mathbb{N}} \, \, n^2 \leq  Cf'(1/n^2) n^\tau \},$$
     where $E=\{ (x,y) \in \mathbb{R}^2 : 0 \leq x \leq 1, \, g(x) \leq y \leq f(x) \}$ .
   \item The techniques we use also work when $k = 0$ or $k = 1$, but in these cases, known results can be used to determine the $L^p$ Markov exponent (see \cite{Kroo}).
   \item For certain functions $f$, a lower bound for $L^p$ Markov exponent ($1 \leq p < \infty$) can be obtained by using the result that says that the exponent in the $L^p$  norm cannot be less than the exponent in the supremum norm (see \cite{BK}).
   \item The work is devoted to two-dimensional sets. Nevertheless, in many cases the problem of finding the $L^p$ Markov exponent for higher dimensional domains can be reduced to a two dimensional situation. Here are some examples:
        \begin{itemize}
        \item[$\ast$] $\{(x,y) \in \mathbb{R}^2 : 0 \leq x \leq 1, \quad ax^k \leq y \leq x^r \} \times [0,1]^m$,
        \item[$\ast$] $\{(x,y) \in \mathbb{R}^{m+1} :  |x| \leq 1, \quad a|x|^k \leq y \leq |x|^r \}$,
        \item[$\ast$] $\{x \in \mathbb{R}^{m} : a \leq |x_1|^k + |x_2|^k + \ldots +  |x_m|^k, \quad |x_1|^r + |x_2|^r + \ldots +  |x_m|^r<1 \}$,
        \item[$\ast$] $\{(x,y,z) \in \mathbb{R}^{3} : 0 \leq x \leq 1, \, ax^k \leq y \leq x^r, \,  Ax+By + C_1 \leq z \leq Ax+By+C_2  \}$,
      \end{itemize}
for every $k,m \in \mathbb{N}$, $r, A, B, C_1, C_2 \in \mathbb{R}$ so that $k \geq 2$, $r \leq k$, $0<a<1$, $C_1< C_2$  and $Ax+By+C_1\geq 0$ if $0 \leq x \leq 1$, $ax^k \leq y \leq x^r$. Here $|x|:= \sqrt{x^2_1 + x^2_2 + \ldots + x^2_m}$.
\end{itemize}

\section*{Acknowledgment}
The author was supported by the Polish National Science Centre (NCN) Opus grant no. 2017/25/B/ST1/00906.







\end{document}